\documentclass[a4paper]{article}
\usepackage{amssymb,amsfonts,amsmath,latexsym,color, amscd} 

\usepackage{url}

\usepackage[dvips]{psfrag,graphicx} 

\newtheorem{theorem}{Theorem}[section]
\newtheorem{lemma}[theorem]{Lemma}
\newtheorem{corollary}[theorem]{Corollary}
\newtheorem{proposition}[theorem]{Proposition}

\newtheorem{definition}[theorem]{Definition}
\newtheorem{example}{Example}[section]
\newenvironment{proof}{{\par\addvspace{0.1cm}\noindent \bf Proof. }}{\hfill$\Box$\par\medskip} 

\newtheorem{conjecture}[theorem]{Conjecture}

\numberwithin{equation}{section}

\def\e{\varepsilon}

\def\R{\Re\mathfrak{e} \,}%\def\R{\Re\mbox{{\rm e}} \,}
\def\vect#1{\mbox{\boldmath $#1$}} % to make symbols in boldmath to express vectors 
\def\RR{\mathbb{R}}
\def\CC{\mathbb{C}}
\def\E#1#2{E_{#1}(#2)}

\def\dom{\Omega}
\def\uon#1{\vect n_{#1}}
\def\Res{\mbox{\rm Res}}

\if0
\fi

\begin{document}

\title{Regularized Riesz energies of submanifolds}

\author{Jun O'Hara\footnote{The first author was supported by JSPS KAKENHI Grant Number 25610014.} and Gil Solanes\footnote{The second author is a Serra H\'unter Fellow and was supported by FEDER-MINECO grants MTM2012-34834,IEDI-2015-00634, PGC2018-095998-B-I00.}}

\date{}

\maketitle

\begin{abstract}Given a closed submanifold, or a compact regular domain, in euclidean space, we consider the Riesz energy defined as the double integral of some power of the distance between pairs of points. When this integral diverges, we compare two different regularization techniques (Hadamard's finite part and analytic continuation), and show that they give essentially the same result. We prove that some of these energies are invariant under M\"obius transformations, thus giving a generalization to higher dimensions of the M\"obius energy of knots.
\end{abstract}
\maketitle



%%%%%%%%%%%%%%%%%%%%%%%%%%%%%%%%%%%%%%%%%%%%%%%%%%%%%%%%%%%%%%%%%%%%%%%%%%%%%%%%%%%%%%
%!!!!!!!!!!!!!!!!!!!!!!!!!!!!!!!!!!!!!!!!!!!!!!!!!!!!!!!!!!!!!!!!!!!!!
\section{Introduction}
%!!!!!!!!!!!!!!!!!!!!!!!!!!!!!!!!!!!!!!!!!!!!!!!!!!!!!!!!!!!!!!!!!!!!!
Let $M\subset\mathbb R^n $ be either a smooth compact submanifold, or a compact regular domain with smooth boundary. We are interested in the {\em Riesz $z$-energy} 
\begin{equation}\label{def_Riesz_energy}
E_M(z)=\int_{M\times M}|x-y|^z\,dxdy,
\end{equation}where $dx,dy$ denote the volume element of $M$. This integral is well-defined if $z>-\dim M$ and  diverges otherwise. In the latter case we apply two techniques from the theory of generalized functions to regularize the divergent integral: {\em Hadamard's finite part} and {\em analytic continuation}. After showing that these two procedures give essentially the same result, we study the properties of the energies thus obtained. In particular, we show that $E_M(-2m)$ is M\"obius invariant if $M$ is a closed submanifold of odd dimension $m$, and also if $M$ is a regular domain in an even dimensional Euclidean space $\mathbb R^m$.

\medskip
To put our results in perspective let us review some background. 
The first author introduced the {\em energy of a knot} $K$ in \cite{O1}, with the aim to produce a canonical representative (the energy minimizer) in each knot type. This energy is given by 
\begin{eqnarray}
E(K)&=&
\displaystyle \lim_{\e\to 0^+}\left(\int_{K\times K\setminus \Delta_\e}\frac{dxdy}{|x-y|^2}-\frac{2L(K)}\e\right),\label{def_energy_knot}
\end{eqnarray}
where
\begin{equation*}%\label{def_Delta_e}
\Delta_\e=\{(x,y)\in\RR^n\times\RR^n\,:\,|x-y|\le\e\}. 
\end{equation*}
This can be viewed as Hadamard's finite part of the divergent integral $\int_{K\times K}|x-y|^{-2}\,dxdy$. Indeed, Hadamard's regularization can be carried out as follows. First one restricts the integration to the complement of some $\e$-neighborhood of the set where the integrand blows up. Then one expands the result in a Laurent series in $\e$ and finally takes the constant term in the series as the {\em finite part} of the integral. Hadamard's finite part can be considered as a generalization of Cauchy's principal value; e.g. they coincide for $\int_{-1}^1\frac{1}{x} dx$ (cf. also \cite[eq. II.2.29]{schwartz}). 

Another approach to $E(K)$ was used by Brylinski \cite{B} who defined the {\em beta function} $B_K(z)$ of a knot $K$  by means of a different regularization method. 
First, given a knot (closed curve) $K\subset\mathbb R^3$, he considered the complex function 
\[B_K(z)= \int_{K\times K}|x-y|^z\,dxdy,\qquad z\in\mathbb C\]
which is holomorphic on the domain $\R z>-1$. 
He then extended this function analytically to  a meromorphic function on the whole complex plane with simple poles at $z=-1,-3,-5,\dots$. Finally, Brylinski showed that $B_K(-2)=E(K)$. 

It turns out that $E(K)$ is invariant under M\"obius transformations (cf. \cite{FHW}), and it is thus often called {\em M\"obius energy}. 
This motivated the search of similar functionals on higher dimensional submanifolds (see \cite{AS, KS}). For closed surfaces $M$ in $\RR^3$, Auckly and Sadun (\cite{AS}) defined the following functional
\begin{eqnarray}
E_{AS}(M)&=&
\lim_{\e\to0^+}\left(\int_{M\times M\setminus\Delta_\e}|x-y|^{-4}dxdy-\frac{\pi A(M)}{\e^2}+\frac{\pi\log\e}8\int_M(\kappa_1-\kappa_2)^2dx\right) \label{Hadamard_reg_surface_energy} \\
&&+\frac{\pi}{16}\int_M(\kappa_1-\kappa_2)^2\log (\kappa_1-\kappa_2)^2dx
+\frac{\pi^2}2\chi(M), \nonumber
\end{eqnarray}
where $\kappa_1$ and $\kappa_2$ are principal curvatures of $M$ at $x$, and $\chi(M)$ is the Euler characteristic. 
The right hand side of \eqref{Hadamard_reg_surface_energy} is Hadamard's finite part of $\int_{M\times M} |x-y|^{-4}\,dxdy$. The additional term $(\pi/16)\int_M(\kappa_1-\kappa_2)^2\log(\kappa_1-\kappa_2)^2\,dx$ was added to make the resulting energy M\"obius invariant, but it is not the only possible choice for this purpose, as was pointed out in \cite{AS}. 

On the other hand, Fuller and Vemuri (\cite{FV}) generalized Brylinski's beta function to closed surfaces and closed submanifolds of Euclidean space in general. For a closed surface $M$, they extended the domain of $B_M(z)=\int_M|x-y|^zdxdy$ by analytic continuation to get a meromorphic function on the whole complex plane with simple poles at $z=-2,-4,-6,\dots$. They showed moreover that the residues of these poles are expressible as integrals of contractions of the second fundamental form of $M$. As for M\"obius invariance, while the integrand $|y-x|^zdxdy$ is a M\"obius invariant $2m$-form for $z=-2m$, it was unclear whether the regularized integral $B_M(-2m)$ would be invariant under M\"obius transformations.

In this paper we begin by showing that Hadamard's finite part of the Riesz energy $E_M(z)$ coincides with the meromorphic function $B_M(z)$ where this function is defined. At the poles, Hadamard's finite part exists and equals the beta function $B_M(z)$ with the pole {\em removed} (see \eqref{Riesz_energy_Hadamard=analytic_continuation}). This extends Brylinski's result to any exponent $z$ and to general dimensions. 
We also give a simple alternative description of the residues of $B_M(z)$ in terms of the volumes of extrinsic spheres of $M$. 

Finally, we show that when $m=\dim M$ is odd, the energy $E_M(-2m)=B_M(-2m)$ is invariant under M\"obius transformations. This gives the desired generalization of the M\"obius energy in the case of odd dimensional submanifolds. For even dimensional submanifolds, we conjecture that none of the energies $E_M(z)$ is M\"obius invariant. We prove this conjecture in the case of two-dimensional surfaces. 

\medskip
The results mentioned so far deal with closed submanifolds, but it makes sense to consider \eqref{def_Riesz_energy} also in the case where $M$ is a compact submanifold with boundary. In particular, we are interested in the case where $M=\dom$ is a compact domain with smooth boundary. For convex domains, the Riesz energy has been considered in \cite{HR} in connection with the statistics of electromagnetic wave propagation inside a domain. Besides, the Riesz energy is closely related to the so-called {\em fractional perimeter} of the domain (cf. e.g. \cite{CRS,L}).

In the last part of the paper, we use the techniques 
mentioned before to regularize the Riesz energy of a smooth regular domain $\dom\subset\mathbb R^n$. In particular we obtain a meromporphic function $B_\dom(z)$ which at the same time is an analyitic continuation of the Riesz energy and of the fractional perimeter (except for a sign). We compute some residues of $B_\dom(z)$ and give some explicit expressions for small dimensions. Finally, we prove that $B_\dom(-2n)$ is invariant under M\"obius transformations if (and only if) the dimension $n$ is even. This generalizes the results obtained by the authors in the planar case (cf. \cite{OS}).

\bigskip
The present version is an integration of the original version, which appeared in Math. Nachr. 291 (2018), 1356-1373, and the errata that gives a corrected proof of Theorem \ref{thm4.11}. 

\bigskip
Acknowledgement: The authors would like to thank Professors Yoshihiro Sawano and Kazushi Yositomi for helpful suggestions. Thanks are also due to the anonymous referees of Mathematische Nachrichten for useful comments.

%!!!!!!!!!!!!!!!!!!!!!!!!!!!!!!!!!!!!!!!!!!!!!!!!!!!!!!!!!!!!!!!!!!!!!
\section{Regularization of divergent integrals}
%!!!!!!!!!!!!!!!!!!!!!!!!!!!!!!!!!!!!!!!!!!!!!!!!!!!!!!!!!!!!!!!!!!!!!
Let us recall two techniques in the theory of generalized functions (or distributions) that are used in the regularization of divergent integrals. The reader is refered to \cite{schwartz,GS} for more details. 

Consider the integral 
\begin{equation}\label{example_regularization}
\int_0^dt^z\,dt,\qquad z\in\CC
\end{equation}
where $d$ is a positive constant. It converges for $\R z>-1$. 

\begin{enumerate}
\item[(i)] For a small positive number $\e$ we have 
\[
\int_\e^dt^z\,dt=\left\{
\begin{array}{lr}
\displaystyle  \frac{d^{z+1}}{z+1}-\frac{\e^{z+1}}{z+1}, & \qquad z\ne-1,\\[4mm]
\displaystyle \log d-\log\e, &z=-1.
\end{array}
\right.
\]
{\em Hadamard's finite part} of \eqref{example_regularization} is defined for every $z\in\CC$ as 
\[
\textrm{Pf.}\int_0^dt^z\,dt
=\left\{
\begin{array}{ll}
\displaystyle \lim_{\e\to0^+}\left(\int_\e^dt^z\,dt+\frac{\e^{z+1}}{z+1}\right)=\frac{d^{z+1}}{z+1} & \>\>\> (z\ne-1),\\[4mm]
\displaystyle \lim_{\e\to0^+}\left(\int_\e^d\frac{dt}{t}+\log\e\right)=\log d& \>\>\> (z=-1).
\end{array}
\right.
\]

\item[(ii)] Consider the complex function 
\[
 f(z)=\int_0^dt^z\,dt,
\]
which is well defined and holomorphic on $\{z\in\CC\colon\R z>-1\}$. 
It extends by analytic continuation to the meromorphic function $f(z)=d^{z+1}/(z+1)$ on the whole complex plane with a simple pole at $z=-1$ with residue $\Res(f,-1)=1$. 
\end{enumerate}

The relation between these two methods is given by 
\begin{equation}\label{zneq1}
f(z)=\mathrm{Pf.}\int_0^d t^z dt\qquad z\neq -1
\end{equation}
\begin{equation}\label{pole_remove}
\lim_{z\to-1}\left(f(z)-\frac1 
{z+1}\right)
=\lim_{z\to -1}\frac{d^{z+1}-1}{z+1}=\log d
=\textrm{Pf.}\int_0^dt^{-1}\,dt.
\end{equation}

More generally, let $\varphi(t)$ be a smooth function, and consider 
\[F(z)=\int_0^d t^z\varphi(t)dt\]
which is well defined if $\R z> -1$. 
For any natural number $k$, the previous integral can be extended to $\R z>-k-1$ as follows. Put 
\begin{eqnarray}
\varphi_{k-1}(t)&=&\displaystyle \sum_{j=0}^{k-1}\frac{\varphi^{(j)}(0)}{j!}t^j, \nonumber
\\[4mm]
h_{z,k}(t)&=&t^{z}\varphi(t)-t^{z}\varphi_{k-1}(t)
\displaystyle =t^z\left[\varphi(t)-\varphi(0)-\varphi'(0)t-\dots -\frac{\varphi^{(k-1)}(0)}{(k-1)!}\,t^{k-1}\right]. \nonumber %\label{laurent} 
\end{eqnarray}
Since $h_{z,k}(t)$ can be estimated by $t^{z+k}$, it is integrable on $[0,d]$ when $\R z>-k-1$. Therefore, the regularization can be reduced to that of 
\begin{equation}\label{regularizationpart}
\int_0^dt^z\varphi_{k-1}(t)\,dt=\sum_{j=1}^{k}\int_0^d\frac{\varphi^{(j-1)}(0)}{(j-1)!}\,t^{z+j-1}\,dt. 
\end{equation}
By setting that the finite part of a convergent integral equals the integral itself, and by linearity, we arrive at the following definition of Hadamard's finite part (cf. \cite[(II,2;26)]{schwartz}) 
\begin{align}
\textrm{Pf.}\int_0^d t^{z}\varphi(t)dt&=\int_0^dh_{z,k}(t)dt+\textrm{Pf.}\int_0^d t^{z}\varphi_{k-1}(t)dt \nonumber %\label{int_laurent} 
\\
&=\lim_{\e\to 0}\left[\int_\e^dt^{z}\varphi(t)dt+\sum_{j=1}^{k}\frac{\varphi^{(j-1)}(0)}{(j-1)!}\frac{\e^{z+j}}{z+j}\right].  \label{Pf}
\end{align}
If $z$ is a negative integer then $\e^0/0$ appears above and is to be replaced by $\log\e$.

On the other hand, since $\int_0^d h_{z,k}(t)dt$ is holomorphic on $z$, equality \eqref{regularizationpart} shows that the integral $F(z)$ can be analytically continued to a meromorphic function on the complex plane which we denote again by $F(z)$. On each half-plane $\R z>-k-1$,  it is given by 
\begin{equation}F(z)
= \int_0^d h_{z,k}(t)dt
%\int_0^dt^z\left[\varphi(t)-\varphi(0)-\varphi'(0)t-\dots -\frac{\varphi^{(k-1)}(0)}{(k-1)!}\,t^{k-1}\right]\,dt 
+\sum_{j=1}^{k}\frac{\varphi^{(j-1)}(0)\,d^{z+j}}{(j-1)!\,(z+j)}.
\label{GS}
\end{equation}
\noindent
This function has (possible) poles at negative integers. The corresponding residues are
\begin{equation}\label{basic_residues}\Res(F,-j)=\frac{\varphi^{(j-1)}(0)}{(j-1)!}.\end{equation}

The relation between these two regularizations can be obtained from \eqref{zneq1}, \eqref{pole_remove}, and \eqref{regularizationpart}. 
When $z$ is not a negative integer, %by \eqref{int_laurent} and \eqref{zneq1},
\begin{equation}
 \label{residue_continuation}\textrm{Pf.}\int_0^dt^z\varphi(t)\,dt=F(z).
\end{equation}
When $z$ is a negative integer $-k$, 
%For $k\in\mathbb N$, by \eqref{zneq1}, \eqref{pole_remove}, and \eqref{int_laurent},
\begin{equation}\label{Hadamard=analytic_continuation}
\textrm{Pf.}\int_0^d{t^{-k}\varphi(t)}\,dt=\lim_{z\to-k}\left(F(z)-\frac{\varphi^{(k-1)}(0)}{(k-1)!(z+k)}\right).
\end{equation}
Note that a $\log$ term appears in \eqref{Pf} exactly when $F(z)$ has a pole in $z$. 

Finally, given an integrable compactly supported function  $\varphi\colon[0,\infty)\to\mathbb R$ which is smooth in some interval $[0,d]$, one defines
\[
 \mathrm{Pf.}\int_0^\infty t^z\varphi(t)dt= \mathrm{Pf.}\int_0^d t^z\varphi(t)dt+\int_d^\infty t^z \varphi(t) dt.
\]
In particular, the integral $\int_0^\infty t^z\varphi(t)dt$, which converges for $\R z>-1$, can be extended to a meromorphic function. 
%!!!!!!!!!!!!!!!!!!!!!!!!!!!!!!!!!!!!!!!!!!!!!!!!!!!!!!!!!!!!!!!!!!!!!
\section{Riesz energies of closed submanifolds}
%!!!!!!!!!!!!!!!!!!!!!!!!!!!!!!!!!!!!!!!!!!!!!!!!!!!!!!!!!!!!!!!!!!!!!
Let $M$ be a closed (i.e. compact without boundary)  submanifold of dimension $m$ in $\RR^n$. We are interested in the following integral
\begin{equation}\label{riesz}
\int_{M\times M}|x-y|^z\,dxdy
\end{equation}
which is absolutely convergent for $\R z>-m$. It was shown, first by Brylinski in the case $m=1$, and then by Fuller and Vemuri for general $m$, that \eqref{riesz} can be extended by analytic continuation to a meromorphic function $B_M(z)$ on the complex plane, called the {\em beta function of $M$}, with possible poles at $z=-m-2j$ where $j\in\mathbb Z, j\geq 0$. It was shown that the residues of these poles are expressible as integrals of complete contractions of the second fundamental form of $M$. Here we provide an alternative, somewhat more concrete, interpretation of these residues. 

Furthermore, we compare the analytic continuation $B_M(z)$ with the following alternative regularization of \eqref{riesz}, based on  Hadamard's finite part regularization. When the integral \eqref{riesz} diverges, one can expand 
\begin{equation*}\label{Riesz_energy_Delta_e}
\int_{M\times M\setminus\Delta_\e}|x-y|^z\,dxdy
\end{equation*}
in a Laurent series (possibly with a $\log$ term) of $\e$. The constant term in the series will be called {\em Hadamard's finite part} of \eqref{def_Riesz_energy}. In case $M=K$ is a knot, the first author used this method to introduce the so-called {\em energy of a knot} (or {\em M\"obius energy}) $E(K)$ (see \eqref{def_energy_knot} or \cite{O1}). It was shown by Brylinski that $E(K)=B_K(-2)$. Here we show similar relations for the other values of the beta function, not only in the case of knots, but also for submanifolds of any dimension.

Furthermore, we show that, for odd dimensional submanifolds, taking $z=-2m$ gives an energy that is invariant under M\"obius transformations. This generalizes the fact that the energy of knots $E(K)$ is M\"obius invariant (cf. \cite{FHW}).

\subsection{Analytic continuation}
Our approach to Riesz energies depends on a careful analysis of the following functions. Define $\psi_{M,x}(t)$ by 
\begin{equation*}\label{def_varphi}
\psi_{M,x}(t)=\textrm{vol}(M\cap B_t(x)),\qquad t\geq 0,
\end{equation*}
where $B_t(x)$ is the ball of center $x$ and radius $t$. The sets $M\cap B_t(x)$ are usually called  {\em extrinsic balls} (cf. e.g. \cite{KP}). 

\begin{proposition}\label{even_odd}
(i) There exists $d>0$ such that, for each $x\in M$ the  function 
\[
\psi_{M,x}(t)=\mathrm{vol}(M\cap B_t(x)), \qquad  0\leq t< d
\]
extends to a smooth function $\varphi(t)$ defined for $t\in(-d,d)$ such that $\varphi(-t)=(-1)^m\varphi(t)$ and $\varphi^{(i)}(0)=0$ for $i<m$.

\medskip
(ii) More generally, given a smooth function $\rho$ on $M\times M$, the same conclusion as above holds for 
\[
 \psi_{\rho,x}(t)=\int_{M\cap B_t(x)} \rho(x,y)dy.
\]
Moreover, if $(\rho_i)_{i=1}^\infty$ is a sequence of smooth functions with derivatives of all orders converging uniformly to $0$, then $\psi_{\rho_i,x}$ and its derivatives also converge uniformly to $0$. 
\end{proposition}
\begin{proof}
$(i)$ It is clear that $\psi_{M,x}(t)$ is smooth at any $t\neq 0$ such that $\partial B_t(x)$ is transverse to $M$. Since $M$ is compact, there is some $d>0$ such that $\partial B_t(x)$  is transverse to $M$ for every $x\in M$ and any $t\in(0,d)$.  Given $x\in M$, take $\varphi(t)=\psi_{M,x}(t)$ for $t\geq 0$, and $\varphi(t)=(-1)^m\psi_{M,x}(-t)$ for $t<0$. We need to show that $\varphi(t)$ is smooth at $t=0$. 

Let us assume for simplicity that $x=0$. Let $\phi\colon\RR^m\to M$ be a coordinate chart with $\phi(0)=x=0$, and let 
\[
 \overline \phi(u,r)=\left(\frac{r}{|r|}\frac{\phi(ru)}{\|\phi(ru)\|},\frac{r}{|r|}\|\phi(ru)\|\right),\qquad u\in S^{m-1}, r\in\mathbb R\setminus\{0\}.
\]
It is shown in \cite{blowup}, that $\overline\phi$ extends to a smooth map $\overline\phi\colon S^{m-1}\times\RR\to S^{n-1}\times \RR$.
For each $u\in S^{m-1}$ let $g_u\colon\RR\to \RR$  be the second coordinate of $\overline\phi(u,\cdot)$. Since $g_u'(0)\neq 0$, the function $g_u$ has a smooth inverse in a neighborhood of $r=0$. Now, for small $t\geq 0$ one has
\[
\psi_{M,x}(t)
=\int_{S^{m-1}}\int_0^{g_u^{-1}(t)}\mathrm{jac}(\phi)_{r\cdot u}r^{m-1}\,dr du.
\]
The right hand side defines a smooth  function of $t$ in a neighborhood of $t=0$, and it coincides with $\varphi(t)=(-1)^m\psi_{M,x}(-t)$ for small negative $t$, since $g_{u}(-r)=-g_{-u}(r)$, and thus $g_{u}^{-1}(-t)=-g_{-u}^{-1}(t)$. Therefore $\varphi(t)$ is smooth at $t=0$ and hence on $(-d,d)$. Moreover, if $1\leq j\leq m-1$, then 
\[
 \frac{d^j}{dt^j}\int_0^{g_u^{-1}(t)}\mathrm{jac}(\phi)_{r\cdot u}r^{m-1}dr=\eta_j(t)(g_u^{-1}(t))^{m-j}
\]
for some smooth function $\eta_j$. Hence, $\varphi^{(j)}(0)=0$. 

\medskip
$(ii)$
The same arguments as in the previous case give 
\[
 \psi_{\rho,x}(t)=\int_{S^{m-1}}\int_0^{g_u^{-1}(t)}\rho(x,\phi(r\cdot u))\mathrm{jac}(\phi)_{r\cdot u}r^{m-1}\,dr du.
\]
Hence, the previous proof applies to $\psi_{\rho,x}(t)$ as well. The last part of the statement follows by uniform convergence. 
\end{proof}
Notice that, by the previous proof, for $\psi_{\rho,x}(t)$ to be smooth around $t=0$ it is in fact enough that $\rho(x,y)|y-x|^{m-1}$ be smooth. Note also that $\psi_{M,x}(t)$ may not be globally smooth, as the case $M=S^{n-1}$ shows.

\bigskip
In the following we denote by $b_{M,k}(x)$ the coefficients of the Taylor series of $\psi_{M,x}(t)$ around $t=0$; i.e.
\[
b_{M,k}(x)=\left.\frac{1}{k!}\frac{d^{k}}{dt^k}\right|_{t=0}
\psi_{M,x}(t).\]

\begin{corollary}\label{coro_odd}
If $k-m$ is odd, then $b_{M,k}(x)=0$. 
\end{corollary}
For small $k$, the coefficients $b_{M,k}(x)$  were given in \cite{KP}. For instance, if $M$ is a knot (closed curve) in $\RR^n$, then
\begin{eqnarray*}
\psi_{M,x}(t)&=&\displaystyle 2t+\frac{\kappa^2}{12}t^3+O(t^5), \label{varphi_knot}
\end{eqnarray*}where $\kappa$ is the curvature of $M$ at $x$.
If $M$ is a closed surface in $\RR^n$, then
\begin{eqnarray}\label{kp_surfaces}
\psi_{M,x}(t)&=&\displaystyle \pi t^2+\frac\pi{32}(2\|B\|^2-\|H\|^2)t^4+O(t^6), \label{varphi_surface}
\end{eqnarray}where $\|B\|$ denotes the Hilbert-Schmidt norm of the second fundamental form $B(X,Y)=(\nabla_XY)^\bot$, and $H=\mathrm{tr}B$ is the mean-curvature vector. In particular, for $n=3$
\[
 2\|B\|^2-\|H\|^2=(\kappa_1-\kappa_2)^2
\]
where $\kappa_1,\kappa_2$ denote the principal curvatures of $M$ at $x$. 

Let 
\[
 \psi_M(t)=\int_{M}\psi_{M,x}(t)\,dx=\int_{(M\times M)\cap \Delta_t} dxdy,
\]
and more generally, given an integrable function $\rho$ on $M\times M$, let
\[
 \psi_\rho(t)=\int_{M}\psi_{\rho,x}(t)\,dx=\int_{(M\times M)\cap \Delta_t}\rho(x,y) dxdy.
\]

\begin{proposition}\label{prop_coarea}
The function $\psi_\rho(t)$ has derivative almost everywhere and 
\[
 \psi_\rho(t)=\int_0^t\psi'_\rho(s)ds.
\]
For $\R z> -m$, 
\begin{equation}  \label{coarea_rho}
\int_{M\times M}|x-y|^z\rho(x,y)dxdy
=\int_0^\infty t^z\psi_{\rho}'(t)dt.
\end{equation}
\end{proposition}
\begin{proof} 
By the coarea formula applied to the function $u(x,y)=|y-x|$ defined on $M\times M$ we have
\begin{align}
\psi_{\rho}(s)&=\int_0^s \left(\int_{u^{-1}(t)}\frac{\rho(x,y)}{|\nabla u(x,y)|}d\mu  \right)dt \nonumber \\
\int_{M\times M}|x-y|^z\rho(x,y)dxdy&=\int_0^\infty t^z\left(\int_{u^{-1}(t)}\frac{\rho(x,y)}{|\nabla u(x,y)|} d\mu \right)dt, \nonumber %\label{eq_coarea2}
\end{align}
where $\nabla$ stands for the gradient in $M\times M$, and $d\mu$ denotes the Riemmannian volume element on $u^{-1}(t)$. Note that, by Sard's theorem, $u^{-1}(t)$ is smooth for almost every $t$, and  the inner intergals define a function which is continuous at almost every $t$. Differentiating the first equation with respect to $s$ yields the result.
\end{proof}

We deduce that \eqref{riesz} extends to a meromorphic function $B_M(z)$ on the complex plane with possible poles at $z=-m-2j$ with $j\in\mathbb Z,j\geq 0$, as shown first by Brylinski and Fuller-Vemuri. The following description of the residues of these poles is new.
\begin{proposition}\label{residues_beta} The meromorphic function $B_M(z)$ has the following residue at $z=-m-2j$
\begin{equation}\label{residue_Riesz_energy}
R_M(-m-2j)=(m+2j)\int_M b_{M,m+2j}(x)dx, \hspace{1cm}j\in\mathbb Z,\ j\ge 0.
\end{equation}
If $\int_M b_{M,k}(x)dx =0$, then $B_M(-k)$ has no pole at $z=-k$. 
\end{proposition}
\proof This follows at once from \eqref{basic_residues}.\endproof

\begin{example}\label{spheres} \rm The beta function of the $n$-dimensional unit sphere is given by (cf. \cite{B,FV})
\[
B_{S^n}(z)=2^{z+n-1}o_{n-1}o_{n}B\left(\frac{z+n}2,\frac{n}2\right),
\]
where $o_k$ is the volume of the unit $k$-sphere in $\RR^{k+1}$, and $B(x,y)$ is Euler's beta function. 
It follows that if $n$ is odd then $B_{S^n}$ has infinitely many poles at $z=-n, -n-2, -n-4, \dots$, and if $n$ is even then $B_{S^n}$ has exactly $n/2$ poles at $z=-n, -n-2, \dots, -2n+2$.
\end{example}

\subsection{Hadamard's finite part}
Next we compare $B_M(z)$ with Hadamard's regularization. 

\begin{definition}\rm 
For any $z\in\mathbb C$ we define
\[
\E{M}{z}
=\mathrm{Pf.}\int_{M\times M}|x-y|^z dydx=\mathrm{Pf.}\int_0^\infty t^z\psi_{M}'(t)dt
\]and call it the {\em regularized $z$-energy} of $M$. 
\end{definition}

Equivalently, if $z$ is not a negative integer, and $\R z> -k-1$ 
for some $k\in\mathbb Z$, then by \eqref{Pf}
\begin{equation*}\label{Hadamard_E_M}
E_M(z)=\lim_{\e\to0^+}\left(\int_{M\times M\setminus\Delta_\e}{|x-y|^{z}}{dxdy}
-\sum_{j=1}^{k}\frac{j}{(-z-j)\e^{-z-j}}\int_M b_{M,j}(x)dx \right).
\end{equation*}
For $z=-k\in\mathbb Z$, by \eqref{Pf} and the explanation after that,
\begin{equation*}\label{Hadamard_E_M2}
\begin{array}{l}
\displaystyle E_M(-k)=
\lim_{\e\to0^+}\left(\int_{M\times M\setminus\Delta_\e}\frac{dxdy}{|x-y|^{k}}

-\sum_{j=1}^{k-1}\frac{j}{(k-j)\e^{k-j}}\int_M b_{M,j}(x)dx +k\log\e\int_M b_{M,k}(x)dx\right).
\end{array}
\end{equation*}
Remark that $b_{M,k}(x)=0$ if $k<m$ by Proposition \ref{even_odd}. 

The relation between Hadamard's finite part and regularization by analytic continuation is given next.
\begin{proposition}\mbox{}
\begin{enumerate}
\item[(i)] $B_M(z)$ can have poles only at $z=-m-2i$ with $i\in\mathbb  Z, i\geq0$. 
\item[(ii)] Away from the poles of $B_M(z)$, analytic continuation and Hadamard's finite part coincide: 
\[
\E{M}{z}
=B_M(z),\qquad z\neq -m,-m-2,-m-4,\ldots
\]
and the same holds for $z=-k$  if $\int_M b_{M,k}(x)dx=0$.

\item[(iii)] If $B_M(z)$ has a pole at $z=-k$, then\begin{equation}\label{Riesz_energy_Hadamard=analytic_continuation}
\E{M}{z}=\lim_{z\to-k}\left(B_M(z)
-\frac{k}{z+k}R_M(-k)\right)=\lim_{z\to-k}\left(B_M(z)
-\frac{k}{z+k}\int_M b_{M,k}(x)dx\right).
\end{equation}
\end{enumerate}
\end{proposition}
\begin{proof}
 (i) follows from Corollary \ref{coro_odd} and Proposition \ref{residues_beta}. (ii) is a consequence of \eqref{residue_continuation}. (iii) follows from \eqref{Hadamard=analytic_continuation} and Proposition \ref{residues_beta}.
\end{proof}

Note in particular that $B_M(-k)=E_M(-k)$ if $k-n$ is odd. 
Next we summarize the situation for knots and surfaces.  
\begin{proposition} Let $K\subset\mathbb R^n$ be a smooth knot (i.e. closed curve). Then
Brylinski's beta function has simple poles at negative odd integers. The residues at $z=-1,-3$ are 
$$R_K(-1)=2 L(K),\qquad R_K(-3)=\frac14\int_K\kappa(x)^2dx.$$ 
The regularized $z$-energies for $z=-1,-3$ are given by 
\[
\begin{array}{rcl}
\displaystyle 
\E{K}{-1}
&=&\displaystyle \lim_{\e\to0^+}\left(\int_{K\times K\setminus\Delta_\e}|x-y|^{-1}dxdy+2L(K)\log\e\right)
=\displaystyle \lim_{z\to-1}\left(B_K(z)-\frac{2L(K)}{z+1}\right),\\[4mm]
\displaystyle 
\E{K}{-3}
&=&\displaystyle \lim_{\e\to0^+}\left(\int_{K\times K\setminus\Delta_\e}|x-y|^{-3}dxdy-\frac{L(K)}{\e^2}+\frac{\log\e}4\int_K\kappa(x)^2dx\right)\\[4mm]
&=&\displaystyle \lim_{z\to-3}\left(B_K(z)-\frac{1}{4(z+3)}\int_K\kappa(x)^2dx\right).
\end{array}
\]
For $\R z> -5, z\neq -1,-3$, it is 
\begin{align*}
E_K(z)&=\lim_{\e\to 0^+}\left(\int_{K\times K\setminus\Delta_\e}|x-y|^{z}dxdy-\frac{2L(K)}{(-z-1)\e^{-z-1}}-\frac{1}{4(-z-3)\e^{-z-3}}\int_K \kappa(x)^2dx\right)\\&=B_K(z).
\end{align*}
\end{proposition}
The residues of $B_K(z)$ for $z=-1,-3,-5$ were computed by Brylinski in \cite{B} (here we took the opportunity to correct the coefficient of $R_K(-3)$ given there) for knots in $\RR^3$. 

A similar Hadamard regularization was used in \cite{O1,O2} to define energy functionals on knots. The approach in this paper is slightly different since our regularized integrals are with respect to $t$, the extrinsic distance, whereas intrinsic distance (i.e. arc-length) was used in \cite{O1,O2}. Still, the resulting energies are closely related. For instance, when $K$ has length 1, the functionals $E^{(\alpha)}(K)$ in \cite[Section 2.2]{O2} are related to $E_K(-\alpha)$ by 
\[
 E^{(\alpha)}(K)=E_K(-\alpha)+\frac{2^\alpha}{\alpha-1},\qquad 1<\alpha<3
\]
\[
E^{(3)}(K) =E_K(-3)+\left(\frac{\log 2}{4}+\frac{1}{12}\right)\int_K\kappa^2(x)dx+4.
\]
The first equality follows from equation (2.17) and Remark 2.2.1 of \cite{O2}, and the second one follows from the last formula in Remark 2.2.1. When $\alpha>3$, the relation is more complicated but can be obtained in a similar way.  
\begin{proposition} Let $M\subset \mathbb R^3$ be a smooth closed surface. The beta function $B_M(z)$ has simple poles at negative even integers. The residues at $z=-2,-4$ are
\[
 R_M(-2)=2\pi A(M),\qquad R_M(-4)=\frac\pi 8\int_M(\kappa_1-\kappa_2)^2dx
\]
The regularized $z$-energy for $z=-2,-4$ is given by 
\[
\begin{array}{rcl}
\E{M}{-2}
&=&\displaystyle \lim_{\e\to0^+}\left(\int_{M\times M\setminus\Delta_\e}|x-y|^{-2}dxdy+2\pi A(M)\log\e\right)
=\displaystyle \lim_{z\to-2}\left(B_M(z)-\frac{2\pi A(M)}{z+2}\right),\\[4mm]
\E{M}{-4}
&=&\displaystyle \lim_{\e\to0^+}\left(\int_{M\times M\setminus\Delta_\e}|x-y|^{-4}dxdy-\frac{\pi A(M)}{\e^2}+\frac{\pi\log\e}8\int_M(\kappa_1-\kappa_2)^2dx\right)\\[4mm]
&=&\displaystyle \lim_{z\to-4}\left(B_M(z)-\frac\pi{8(z+4)}\int_M(\kappa_1-\kappa_2)^2dx\right).
\end{array}
\]
For $\R z> -6, z\neq -2,-4$, they are 
\begin{align*} \E{M}{z}
 &=\displaystyle \lim_{\e\to0^+}\left(\int_{M\times M\setminus\Delta_\e}|x-y|^{z}dxdy-\frac{2\pi A(M)}{(-z-2)\e^{-z-2}}-\frac{\pi}{8(-z-4)\e^{-z-4}}\int_M(\kappa_1-\kappa_2)^2dx\right)\\
 &=\displaystyle B_M(z).\end{align*}
\end{proposition}
{The residues of $B_M(z)$  for $z=-2,-4,-6$ were obtained by Fuller and Vemuri in \cite{FV} (we corrected the coefficient of $R_M(-4)$). Using their expression for $R_M(-6)$  one can extend the previous formulas to $\R z> -8$.
%!!!!!!!!!!!!!!!!!!!!!!!!!!!!!!!!!!!!!!!!!!!!!!!!!!!!!!!!!!!!!!!!

\subsection{M\"obius invariance}
Here we study the M\"obius invariance of these energies. We first discuss scale-invariance. 

\begin{lemma}\label{lemma_residue_homothety}
Under a homothety $x\mapsto cx$ $(c>0)$, the residues of the beta function behave as follows 
\begin{equation*}\label{residue_homothety}
R_{cM}(-k)=c^{2m-k} R_M(-k) \hspace{0.7cm}(k\ge m).
\end{equation*}
\end{lemma}
%!!!!!!!!!!!!!!!!!!!!!!!!!!!!!!!!!!!!!!!!!!!!!!!!!!!!!!!!!!!!!!!!
%
\begin{proof} We have the following Taylor series expansions 
\[
\begin{array}{l}
\displaystyle \textrm{Vol}(cM\cap B_{ct}(cx))\sim\sum_{k\ge m} b_{cM,k}(cx)\cdot(ct)^k 
,\\[2mm]
\displaystyle c^m \textrm{Vol}(M\cap B_{t}(x))
\sim\sum_{k\ge m} c^mb_{M,k}(x) t^k, 
\end{array}
\]
which implies $b_{cM,k}(cx)=c^{m-k}b_{M,k}(x)$. The conclusion follows from \eqref{residue_Riesz_energy}. 
\end{proof}
%
%!!!!!!!!!!!!!!!!!!!!!!!!!!!!!!!!!!!!!!!!!!!!!!!!!!!!!!!!!!!!!!!!
\begin{proposition}\label{proposition_energy_homothety}
Under a homothety $x\mapsto cx$ $(c>0)$, the regularized $z$-energy behaves as follows 
\begin{equation*}\label{Riesz-energy_homothety}
\E{cM}{z}
=c^{2m+z}\left(\E{M}{z}+(\log c) R_M(z)\right),
\end{equation*}
where $R_M(z)$ is the residue at $z$ if $B_M$ has a pole there, and $R_M(z)=0$ otherwise. Hence the regularized $z$-energy of $m$-dimensional submanifolds is not scale invariant if $z\ne -2m$. 
The regularized $(-2m)$-energy is scale invariant if and only if $R_M(-2m)$ vanishes for any $m$-dimensional $M$. 
\end{proposition}
%!!!!!!!!!!!!!!!!!!!!!!!!!!!!!!!!!!!!!!!!!!!!!!!!!!!!!!!!!!!!!!!!
\begin{proof} Lemma \ref{lemma_residue_homothety} implies 
\[
\begin{array}{rcl}
\E{cM}{z_0}
&=&\displaystyle \lim_{z\to z_0}\left(B_{cM}(z)-\frac{R_{cM}(z_0)}{z-z_0}\right) \\[4mm]
&=&\displaystyle \lim_{z\to z_0}\left(c^{2m+z}\,B_M(z)-\frac{c^{2m+z_0}R_{M}(z_0)}{z-z_0}\right)\\[4mm]
%
% &=&\displaystyle \lim_{z\to z_0}c^{2m+z}\left(\mathrm{Pf.}\int_{M\times M}|x-y|^zdxdy-\frac{c^{z_0-z}-1+1}{z-z_0}R_{M}(z_0)\right).\\[4mm]
%
&=&\displaystyle \lim_{z\to z_0}c^{2m+z}\left(B_M(z)-\frac{R_{M}(z_0)}{z-z_0}+\frac{c^{z_0-z}-1}{z_0-z}R_{M}(z_0)\right).\\[4mm]
\end{array}
\]
Since $\lim_{w\to 0}{(c^w-1)}/w=\log c$, the conclusion follows. 
\end{proof}
In particular, if $M\subset \RR^3$ is a surface, then
\begin{equation}\label{regularized_-4energy_homothety}
\E{cM}{-4}=\E{M}{-4}
+\frac{\pi\log c}8\int_M(\kappa_1-\kappa_2)^2dx,
\end{equation}
and similarly for surfaces in $\mathbb R^n$ (cf. \eqref{kp_surfaces}).
Hence $E_M(-4)$ is not scale invariant unless $M$ is a sphere. This corrects a statement in the conclusion of \cite{FV}.

However, if $M$ is a closed submanifold of odd dimension $m$, then  $E_{M}(-2m)$ is scale invariant. In fact it is M\"obius invariant as we show next. 
\begin{proposition} \label{moebius_invariance}
If $m=\dim M$ is odd, then $E_M(-2m)=E_{I(M)}(-2m)$ for any M\"obius transformation $I$ such that $I(M)$ remains compact. 
\end{proposition}
\begin{proof} Since $E_M(-2m)$ is translation and scale invariant, we can suppose $0\not \in M$, and we only need to prove the statement when $I$ is an inversion in the unit sphere. 
Let $\widetilde M=I(M), \tilde x, \tilde y$ denote the images by $I$ of $M,x,y$ respectively. 
Since 
\[
|\tilde x-\tilde y|=\frac{|x-y|}{|x|\,|y|},\quad d\tilde x=\frac{dx}{|x|^{2m}}, \quad d\tilde y=\frac{dy}{|y|^{2m}},
\]
we have for $\R z>-m$
\[
\int_{\tilde M\times\tilde M}|\tilde x-\tilde y|^z d\tilde xd\tilde y
=\int_{M\times M}|x-y|^z \frac1{|x|^{z+2m}|y|^{z+2m}} dxdy.
\]
Hence, for $\R z>-m$, and using \eqref{coarea_rho}
\begin{align}\notag
 B_{\widetilde M}(z)- B_{ M}(z)
&= \int_{M\times M}|x-y|^z \left[\left(\frac1{|x|\,|y|}\right)^{z+2m}-1\right] dxdy\\
&=\int_0^d t^z  \psi_{\rho_z}'(t) dt+\int_{M\times M\setminus \Delta_d} |x-y|^z \rho_{z}(x,y)dxdy,\label{diffBbigz}
\end{align}
where $\rho_z(x,y)=\left(\frac1{|x|\,|y|}\right)^{z+2m}-1$, and $d>0$ is such that the spheres $\partial B_t(x)$ are transverse to $M$ for all $x$ in $M$ and all $t\in (0,d]$. 
Let
\[
%\Psi(t)=
\Psi_z(t)=\psi_{\rho_z}(t)=\int_{M\times M\cap \Delta_t}\left[\left(\frac1{|x|\,|y|}\right)^{z+2m}-1\right]\,dxdy,
\]
which is smooth in $[0,d]$. 

By putting $\varphi(t)=\Psi_z'(t)$ and $k=2m$ in \eqref{GS}, we can extend the domain of \eqref{diffBbigz} to $\R z>-2m-1$ to obtain 
% implies that, after analytic continuation, 
\begin{align}\label{Psi}
B_{\widetilde M}(z)- B_{ M}(z)&=\displaystyle \int_0^d t^z\left[\Psi_z'(t)-\sum_{j=0}^{2m-1}\frac{\Psi_z^{(j+1)}}{j!}\,t^j
\right]\,dxdt  +\sum_{j=1}^{2m}\frac{\Psi_z^{(j)}(0)\, d^{z+j}}{(j-1)!\,(z+j)}\\
&+\int_{M\times M\setminus \Delta_d} |x-y|^z \rho_{z}(x,y)dxdy. \nonumber %\label{Psibis}
\end{align}
We show next that the three terms in the right hand side of the previous equality converge to 0 as $z$ approaches $-2m$. For the last term, this is true since $\rho_z$ converges uniformly to $0$ on $M\times M$ as $z\to-2m$.  

Since all the derivatives of $\rho_z$ also converge uniformly to $0$ as $z\to-2m$, by Proposition \ref{even_odd} we have
\begin{equation}\label{eq_lim_sup}
\lim_{z\to-2m} \sup_{0\le t\le d}|\Psi_z^{(i)}(t)|=0,\qquad \forall i.
\end{equation}
Since $m$ is odd, we know that $\Psi_z^{(2m)}(0)=0$ for any $z$, so the sum in the last term of line \eqref{Psi} runs over $1\le j\le 2m-1$. By \eqref{eq_lim_sup}, we deduce that this sum goes to $0$ as $z\to-2m$. 

Finally, as
\[
 \left|\int_0^d t^z\left[\Psi_z'(t)-\sum_{j=0}^{2m-1}\frac{\Psi_z^{(j+1)}}{j!}\,t^j
%-\Psi_z'(0)-\Psi_z''(0)t-\dots -\frac{\Psi_z^{(2m)}(0)}{(2m-1)!}\,t^{2m-1}
\right]\,dxdt\right| \displaystyle\leq \sup_{0\le t\le d}\left|\Psi_z^{(2m+1)}(t)\right|\frac{1}{(2m)!}\int_0^d\,t^{z+2m}\,dt,
\]
using \eqref{eq_lim_sup} once more, we see that the first term on the right hand side of \eqref{Psi} goes to $0$ as $z$ approaches $-2m$. This completes the proof.
\end{proof}
\begin{conjecture}The regularized energy $E_M(-2m)$ is not scale invariant if $m=\dim M$ is even; i.e. there exists $M$ such that $R_M(-2m)\ne0$ if $m$ is even.
\end{conjecture}
In particular we conjecture that $E_M(z)$ is a M\"obius invariant only if  $z=-2m$ and $m=\dim M$ is odd. Note that the case of spheres discussed in Example \ref{spheres} does not help in proving this conjecture. The conjecture holds for surfaces in $\RR^3$ by \eqref{regularized_-4energy_homothety}.

%!!!!!!!!!!!!!!!!!!!!!!!!!!!!!!!!!!!!!!!!!!!!!!!!!!!!!!!!!!!!!!!!!!!!!
\section{Energy of regular domains}
%!!!!!!!!!!!!!!!!!!!!!!!!!!!!!!!!!!!!!!!!!!!!!!!!!!!!!!!!!!!!!!!!!!!!!
%
Next, we study the Riesz energies of compact domains with smooth boundary. As before, we regularize when necessary to get a meromorphic function which we call the beta function of the domain. We compute some residues and  give some explicit presentations in low dimensions. Finally we prove that M\"obius invariant regularized Riesz energies exist in even dimensional spaces. 
\subsection{Riesz energies}
Let $\dom$ be a compact domain in $\RR^n$ with smooth boundary $M=\partial\dom$, and $\uon{x}$ the outer unit normal to $\dom$ at a point $x$ in $ M$. 
For $z>-n$, we consider 
\[E_\dom(z)=\int_{\dom\times\dom}|x-y|^z\,dxdy.\]
A closely related quantity is
\[
 P_\dom(z)=\int_{\dom\times\dom^c}|x-y|^z\,dxdy.
\]
This integral converges for $-n-1<\R z<-n$, and is called {\em fractional perimeter} especially when $z\in\RR$ (cf. \cite{CRS}). 

%!!!!!!!!!!!!!!!!!!!!!!!!!!!!!!!!!!!!!!!!!!!!!!!!!!!!!!!!!!!!!!!!
\begin{lemma}\label{lemma_Riesz_energy_compact_bodies_boundary_integral}
For $\R z>-n$ and $z\neq -2$, the Riesz $z$-energy can be expressed by a double integral over the boundary: 
\begin{equation}\label{eq_compact_bodies_z-energy_boundary_integral}
E_\Omega(z)
=\frac{-1}{(z+2)(z+n)}\int_{ M\times M}|x-y|^{z+2}\langle\uon{x},\uon{y}\rangle\,dxdy.
\end{equation}
\end{lemma}
%!!!!!!!!!!!!!!!!!!!!!!!!!!!!!!!!!!!!!!!!!!!!!!!!!!!!!!!!!!!!!!!!
%
\begin{proof}
Since
\[
\begin{array}{rcl}
\textrm{div} _x\left[\,|x-y|^z(x-y)\,\right]
=\displaystyle \sum_{i=1}^n\frac{\partial}{\partial x_i}\left[\,(x_i-y_i)|x-y|^z\,\right]
=(z+n)|x-y|^z,
\end{array}
\]
we have
\begin{align}
\int_{\dom\times\dom}|x-y|^z\,dxdy
&=\displaystyle \frac1{z+n}\int_{\dom}\int_{ M} \langle x-y,\uon{x}\rangle |x-y|^zdx dy\label{first_step}. \nonumber
\end{align}
Similarly, since
\[
\textrm{div} _y\left[|x-y|^{z+2}\uon{x}\right]
=\langle \nabla_y|x-y|^{z+2}, \uon{x}\rangle+|x-y|^{z+2} \textrm{div} _y\uon{x}
=(z+2)\langle |x-y|^{z}(y-x), \uon{x}\rangle,
\]
we find
\begin{align}
 \int_{\dom\times \dom}|x-y|^z\,dxdy&= \frac1{z+n}\int_{ M}\int_{\dom} \langle x-y,\uon{x}\rangle |x-y|^zdydx \nonumber \\
&= \frac{-1}{(z+2)(z+n)}\int_{ M}\int_{ M}|x-y|^{z+2}\langle\uon{x},\uon{y}\rangle dy dx. \nonumber
\end{align}
\end{proof}
With a similar argument one shows that for $-n-1<\R z<-n$,
\[
 P_\dom(z)=\frac{1}{(z+2)(z+n)}\int_{ M}\int_{ M}|y-x|^{z+2}\langle\uon{x},\uon{y}\rangle dy dx.
\]
i.e. the right hand side of \eqref{eq_compact_bodies_z-energy_boundary_integral} gives $-P_\dom(z)$ if $-n-1<\R z<-n$ and $E_\dom(z)$ when $\R z>-n$.

\subsection{Regularization}
In order to extend $E_\Omega(z)$ to the whole complex plane we follow a  similar but not identical procedure as in the case of closed submanifolds.

\smallskip
We will use the following elementary fact.
\begin{lemma}
 Let $A,C\subset N$ be compact domains with regular boundary in a smooth orientable manifold $N$. Suppose that $\partial A,\partial C$ are transverse hypersurfaces. Let $X$ be a complete vector field in $N$ with associated flow $\phi\colon\RR\times N\to N$, and denote $\phi_t=\phi(t,\cdot)$. Then, for any differential form $\omega$ of top degree in $N$, we have
 \begin{equation}\label{eq_lemma}
  \left.\frac{d}{dt}\right|_{t=0}\int_{C\cap \phi_t(A)}\omega=\int_{C\cap\partial A} X\lrcorner\, \omega. 
\end{equation}
\end{lemma}

\begin{proof}Let $I=(-\epsilon,\epsilon)$ such that $\phi_t(\partial A)$ is transverse to $\partial C$ for all $t\in I$.
Consider the vector field $\overline X=(1,X)$ on $I\times N$, which has associated flow ${\overline\phi(t,h,x)=}\overline\phi_t(h,x)=(t+h,\phi_t(x))$. Clearly $\overline X$  is tangent to the hypersurface ${\overline\phi}(I{\times\{0\}}\times\partial A)\subset I\times N$.

Let us take  a vector field on $I\times N$ of the form $\overline Y=(1,Y)$, and such that $\overline Y$ is tangent to  $I\times \partial C$ and also to ${\overline\phi}(I{\times\{0\}}\times\partial A)$. This is possible because these hypersurfaces are transverse.

Given $\epsilon>0$, there is some field $\overline Z_\epsilon=(1,Z_\epsilon)$, tangent to ${\overline\phi}(I{ \times\{0\}}\times\partial A)$, and such that 
\[
 Z_\epsilon(t,x)=Y(t,x),\quad \forall x\in \partial C,\qquad Z_\epsilon(t,x)=X(x),\quad \forall x\notin {(\partial C)_\epsilon }
\]
where ${(\partial C)_\epsilon}$ is the set of points at distance $\le\epsilon$ from ${\partial}C$, with respect to an auxiliary riemannian metric. The flow $\psi_t^\epsilon\colon I\times N\to I\times N$ associated to $\overline Z_\epsilon$ fulfills
\[
 \psi_t^\epsilon(0,A)=(t,\phi_t(A)),\quad \psi_t^\epsilon(0,C)=(t,C).
\]
Hence,
\begin{align*}
\left.\frac{d}{dt}\right|_{t=0}\int_{\phi_t(A)\cap C} \omega&=\left.\frac{d}{dt}\right|_{t=0}\int_{\psi_t^\epsilon(0,A\cap C)} \omega=\left.\frac{d}{dt}\right|_{t=0}\int_{(0,A\cap C)} (\psi_t^\epsilon)^*\omega=\int_{(0,A\cap C)} \mathcal L_{\overline Z_\epsilon}\omega\\
&=\int_{(0,A\cap C)} d(\overline Z_\epsilon\lrcorner \omega)=\int_{(0,C\cap\partial A)} \overline Z_\epsilon\lrcorner \omega+\int_{(0,A\cap\partial C)} \overline Z_\epsilon\lrcorner \omega.
\end{align*}
where $\omega$ denotes also its pull-back to $I\times N$, and $\mathcal L$ is the Lie derivative. 
{The fourth equality follows from Cartan's formula and the fact that $d\omega=0$. }
The second term in the last expression vanishes, since $Z_\epsilon$ is tangent to $\partial C$. Finally, we can assume $Z_\epsilon$ is uniformly bounded and thus
\[
\lim_{\epsilon\to 0} \int_{(0,C\cap\partial A)} \overline Z_\epsilon\lrcorner \omega= \int_{(0,C\cap\partial A)} \overline Z_0\lrcorner \omega=\int_{C\cap\partial A} X\lrcorner \omega .
\]
\end{proof}

\bigskip

Given $\rho\in C(\Omega\times\Omega)$, let
\[
 \Psi_\rho(t)=\int_{(\Omega\times\Omega)\cap\Delta_t} \rho(x,y)dxdy,
\]where $dx,dy$ are volume elements in $\RR^n$, and $\Delta_t=\{(x,y)\colon |y-x|\leq t\}$.
 Put
\begin{align*}
D_t&=\{(x,v)\in \Omega\times S^{n-1}\colon x+tv\in\Omega\},\\
E_t&=\{(x,v)\in M\times S^{n-1}\colon x+tv\in\Omega\}.
\end{align*}

>From here on, let $d>0$ be such that $\partial B_t(x)$ is transverse to $M$ whenever $x\in M$ and $0<t \leq d$. For $0<t<d$, the set $E_t$ is diffeomorphic to the product of $M$ and a closed hemisphere, and $D_t$ is the intersection of two domains with regular boundary such that the two boundaries intersect transversely. 
To see the latter, consider the involution $i_t(x,v)=(x+tv,-v)$ on $\RR^n\times S^{n-1}$, and note that $D_t=(\Omega\times S^{n-1})\cap i_t(\Omega\times S^{n-1})$. 

Given $f\in C(\Omega\times S^{n-1}\times [0,d\,])$, put
\[
 \Phi_f(t)=\int_{D_{t}} f(x,v,t)dxdS^{n-1}_v,\qquad \Xi_f(t)=\int_{E_t} f(x,v,t)dM_xdS^{n-1}_v
\]
where $dM$ and $dS^{n-1}$ are the volume elements in $M$ and $S^{n-1}$ respectively.

\begin{proposition}\label{variation}
Let  $\rho\in C^\infty(\Omega\times\Omega)$ and $f\in C^\infty(\Omega\times S^{n-1}\times [0,d\,])$. For $t_0\in(0,d\,]$, 
\begin{align}
\label{var1} \Psi_\rho'(t_0)=&\int_{D_{t_0}} \rho(x,x+t_0v)t_0^{n-1}dxdS^{n-1}_v\\
\label{var2} \Phi_f'(t_0)=&\int_{D_{t_0}} \frac{\partial f}{\partial t}(x,v,t_0)dxdS^{n-1}_v+\int_{E_{t_0}} \langle \uon{x},v\rangle f(x+t_0v,-v,t_0) dM_xdS^{n-1}_v\\
\label{var3} \Xi_f'(t_0)=&\int_{E_{t_0}} \frac{\partial f}{\partial t}(x,v,t_0)dM_xdS^{n-1}_v  \\
\label{var4} &-\left.\frac{d}{dt}\right|_{t=t_0}\int_{(M\times M)\cap(\Delta_t \setminus\Delta_0)} \frac{\langle \uon{y},{y-x}\rangle}{|y-x|^{n}} \,f \!\left(x,\frac{y-x}{|y-x|},|y-x|\right) dM_xdM_y,  
\end{align}
where $\uon{x}$ is the outer unit normal to $M$ at $x$. 
\end{proposition}
\begin{proof}
Equation \eqref{var1} follows from Fubini's theorem using polar coordinates for $y$ around $x$, and noting that $$(\Omega\times\Omega)\cap\Delta_t=\{(x, x+sv) \colon 0\le s\le t, (x,v)\in D_s\}.$$
 
By the chain rule, to prove \eqref{var2} and \eqref{var3} we only need to consider the case $f(x,v,t)\equiv f(x,v,t_0)$. 
Consider the vector field $X(x,v)=(-v,0)$ on  $\RR^n\times S^{n-1}$ and its associated flow $\phi_t(x,v)=(x-tv,v)$. 
Since
\[
D_{t+s}=(\Omega\times S^{n-1})\cap\phi_s(i_t(\Omega\times S^{n-1}))
\]
and
\[
(\Omega\times S^{n-1})\cap \partial(i_t(\Omega\times S^{n-1}))
=(\Omega\times S^{n-1})\cap i_t(M\times S^{n-1})
=i_t(E_t),
\]
the previous lemma implies
\[\begin{array}{rcl}
\displaystyle \frac{d}{dt}\int_{D_t} f(x,v,t_0)dxdS^{n-1}_v
&=&\displaystyle \left.\frac{d}{ds}\right|_{s=0}\int_{D_{t+s}} f(x,v,t_0)dxdS^{n-1}_v \\[4mm]
&=&\displaystyle \int_{(\Omega\times S^{n-1})\cap \partial(i_t(\Omega\times S^{n-1}))} f(x,v,t_0) X\lrcorner (dx\wedge dS^{n-1}_v) \\[4mm]
&=&\displaystyle \int_{i_t(E_t)} f(x,v,t_0) X\lrcorner (dx\wedge dS^{n-1}_v), 
\end{array}\]
where $dx\wedge dS^{n-1}$ is the differential form corresponding to the measure $dxdS^{n-1}$ with the product orientation. 
Since $i_t^*(dx\wedge dS^{n-1})=(-1)^n dx\wedge dS^{n-1}$, and ${i_t}_*(X)=-X$, 
taking suitable orientations, the previous integral equals
\[
\int_{E_t} f(x+tv,-v,t_0) X\lrcorner (dx\wedge dS^{n-1}_v)=\int_{E_t} \langle \uon{x},v\rangle f(x+tv,-v,t_{0}) dM_xdS^{n-1}_v,
\]
which yields \eqref{var2}.

To prove \eqref{var3}, let $\pi\colon (M\times M)\setminus \Delta_0\to M\times S^{n-1}$  be given by $\pi(x,y)=(x,\frac{y-x}{|y-x|})$. A simple computation shows
\[
 \pi^*(dM\wedge dS^{n-1})_{(x,y)}=\frac{1}{|y-x|^{n-1}} \langle \uon{y},\frac{y-x}{|y-x|}\rangle dM_x\wedge dM_y.
\]
On the other hand, given $(x,v)\in M\times S^{n-1}$  we have
\[
 \sum_{(x,y)\in \pi^{-1}(x,v){\cap(\Delta_{t+h}\setminus\Delta_t)}} \mathrm{sgn}\langle y-x,\uon{y}\rangle=\mathbf{1}_{E_t}(x,v)-\mathbf{1}_{E_{t+h}}(x,v).
\]
Therefore, 
\begin{multline*}
 \int_{(M\times M)\cap(\Delta_{t+h}\setminus\Delta_t)} \frac{\langle \uon{y},{y-x}\rangle}{|y-x|^{n}} f(x,\frac{y-x}{|y-x|},|y-x|) dM_xdM_y\\
 =\int_{E_{t}}f(x,v,t)dM_x\wedge dS_v^{n-1}-\int_{E_{t+h}}f(x,v,t)dM_x\wedge dS_v^{n-1}.
\end{multline*}
This yields \eqref{var3}.
\end{proof}

Given an $n$-dimensional manifold $X$, and $k\in\mathbb N\cup\{\infty\}$,  we take on $C^k(X)$  the structure of locally convex topological vector space defined by the family of seminorms $\| \cdot \|_{\alpha,\phi,K}$ given by
\[
 \| f \|_{\alpha,\phi,K}=\sup_{x\in K}\left|D^\alpha (f\circ \phi)(x)\right|
\]
where $\alpha=(a_1,\ldots,a_n), a_i\in \mathbb N,$
\[
D^\alpha= \frac{\partial^{|\alpha|} }{\partial_{x_1}^{a_1}\cdots \partial_{x_n}^{a_n}},\qquad  |\alpha|=a_1+\cdots +a_n\leq k,
\]
and $(U,\phi)$ is a local chart with $K\subset U$ compact. By \cite[Chapter III \S1.1]{schaefer}, a linear map $L\colon C^\infty(X)\to C^k([0,d\,])$, is continuous if and only if for every $r\in \mathbb N, r\le k$, there exist $c>0$, local charts $\{(U_i,\phi_i)\}_{i=1}^m$, compact sets   $K_i\subset U_i$, and index sets $\alpha_i$  such that
\[
\sup_{t\in[0,d]}\left|\frac{d^rL(f)}{dt^r} (t)\right|<c \sum_{i=1}^m \|f\|_{\alpha_i, \phi_i,K_i}  
\]
for all $f\in C^\infty(X)$.

\begin{proposition}\label{smooth}
 Given $h\in C^\infty(M\times S^{n-1}\times [0,d\,])$, consider 
 \[
  \Lambda_h(t)=\int_{(M\times M)\cap \Delta_t}  h(x,\frac{y-x}{|y-x|},|y-x|) dM_xdM_y.
 \]
Then $\Lambda_h(t)$ is smooth on $[0,d\,]$. Moreover, the map $h\mapsto \Lambda_h$ is continuous with respect to the $C^\infty$-topologies. 
\end{proposition}
\begin{proof}
Away from $t=0$, the statement is easy. In the following we assume $t$ small enough. By compactness, there exist $\epsilon,\delta>0$, and a finite collection of local charts $\phi_i\colon U_i\to M$, and open sets $V_i\subset U_i$ such that $\bigcup_i \phi_i(V_i)=M$,
\[
  B_p(\epsilon)\subset U_i,\qquad\mbox{and}\qquad B_{\phi_{i}(p)}(\delta)\cap M\subset \phi_{i}(B_p(\epsilon)),\qquad\forall p\in V_i, \, {\forall i}.
\]
Using partitions of unity, we can assume that $h$ has support inside $\phi_{i}(V_i)\times S^{n-1}\times [0,d\,]$ for some $i$. 
{Put $\phi=\phi_i, U=U_i$ and $V=V_i$ in what follows. }
Let 
\[
  F(p,u,r)=\frac{\phi(p+ru)-\phi(p)}{|\phi(p+ru)-\phi(p)|},\quad g(p,u,r)=|\phi(p+ru)-\phi(p)|,
 \]
which extend to smooth mappings on ${V}\times S^{n-2}\times [0,\epsilon)$ (cf. \cite{blowup} {Cor. 2.6}).  Since $\frac{\partial g}{\partial r}(p,u,0)> 0$, there exists a smooth function $s(p,u,\tau)$ such that $g(p,u,s(p,u,t))=t$ defined for small $t\geq 0$. Hence, for $0 \leq t \leq\delta$, 
\begin{equation}\label{integrals}
  \Lambda_h(t)= \int_{K}\int_{S^{n-2}}\int_0^{s(p,u,t)} (\mathrm{jac}\phi)_p(\mathrm{jac}\phi)_{p+ru} h(\phi(p),F(p,u,r),g(p,u,r))r^{n-2} dr dS^{n-2}_u dp,
\end{equation}
where $K\subset V$ is the inverse image by $\phi$ of the projection to $M$ of the support of $h$. Since $K$ is compact, the integrand and all its derivatives are uniformly bounded. It follows that the innermost integral defines a smooth function of $t$, and all its derivatives are uniformly bounded. Therefore $\Lambda_h$ is smooth.  Finally, the continuity follows using again partitions of unity and \eqref{integrals}. 
\end{proof}

\begin{proposition}\label{coro}
 Given $f\in C^\infty(M\times S^{n-1}\times [0,d\,])$, the function
 \[
  \Gamma_f(t)=\int_{(M\times M)\cap\Delta_t} \frac{\langle \uon{y},{y-x}\rangle}{|y-x|^{n}} f(x,\frac{y-x}{|y-x|},|y-x|) dM_xdM_y
 \]
is smooth on $[0,d\,]$. The operator $f\mapsto \Gamma_f\in C^\infty([0,d\,])$ is continuous. 
\end{proposition}
\begin{proof} 
Given $f$, take $h(x,v,r)=\langle  \uon{x+{rv}},v\rangle f(x,v,r)$ (after suitably extending $\uon{}$ to a neighborhood of $M$). Arguing as in Proposition 3.3 we have
\begin{equation}\label{preparation}
\Gamma_f(t)=\int_{(M\times M)\cap\Delta_t} {|y-x|^{1-n}} h(x,\frac{y-x}{|y-x|},|y-x|) dM_xdM_y= \int_0^t {\tau}^{1-n} \Lambda'_h({ \tau})\,d{\tau .}
\end{equation} 
By \eqref{integrals}, and using $\langle \uon{y},\frac{y-x}{|y-x|}\rangle=O(|y-x|)$, we have $\Lambda_h(t)=O(t^{n})$. 
It follows by Proposition \ref{smooth} that $\Gamma_f(t)$ is smooth in $[0,d]$. The continuity is also clear by Proposition \ref{smooth}.
\end{proof}

\begin{proposition}\label{uniform}
Given $f\in C^\infty(\Omega\times S^{n-1}\times [0,d\,])$ the functions $\Psi_f$, $\Phi_f$, and $\Xi_f$ are smooth on $[0,d\,]$. The linear maps $f\mapsto \Psi_f,\Phi_f,\Xi_f$ are smooth with respect to the $C^\infty$-topologies.
\end{proposition}
\begin{proof}We will proceed by induction. First,  as for $\Xi_f$, it is clear that $\Xi_f(t)$ is continuous on $[0,d\,]$, and that $f\mapsto \Xi_f\in C([0,d\,])$ is continuous.  
Assume for every $k\leq k_0$ and for any $f\in C^\infty(\Omega\times S^{n-1}\times [0,d\,])$ that  $\Xi_f\in C^k([0,d])$ and the map $f \mapsto \Xi_f$ from $C^\infty(\Omega\times S^{n-1}\times [0,d\,])$ to $C^k([0,d\,])$ is continuous with the given topologies.
 By \eqref{var3} and Proposition \ref{coro} applied to \eqref{var4},  it follows that the same holds for $k=k_0+1$, and by induction for all $k$. 

Proceeding analogously  and using \eqref{var2}, one proves that $\Phi_f(t)$ is smooth and $f\mapsto \Phi_f$ is continuous.  The stated properties of $\Psi_f$ follow by \eqref{var1}.
\end{proof}

\begin{corollary}
The function $$\psi_\dom(t)=\int_{(\dom\times\dom)\cap\Delta_t} dxdy$$  is smooth on $[0,d]$. Moreover,
\begin{equation}\label{expansion_psi}
 \psi_\dom(t)=\frac{o_{n-1}}{n}t^{n}V(\dom)-\frac{o_{n-2}}{(n+1)(n-1)}t^{n+1} A( M)+O(t^{n+3}).
\end{equation}
\end{corollary}
\begin{proof}
 The smoothness follows from Proposition  \ref{uniform}, while the given expansion is a consequence of \eqref{var1} and \eqref{var2}.
\end{proof}

Given any $z\in\CC$, by the coarea formula, one shows as in Proposition \ref{prop_coarea},
\[
 \int_{\Omega\times\Omega\setminus\Delta_\e}|x-y|^z dxdy=\int_\e^\infty t^z\psi_{\dom}'(t) dt,\qquad \e>0.
\]
\begin{definition}
For any $z\in\CC$, the {\em regularized $z$-energy} of a domain $\dom\subset\mathbb R^n$ with smooth boundary is
\begin{equation}\label{eq_def_E_dom}
 E_\dom(z)=\mathrm{Pf.}\int_0^\infty t^z\psi'_\dom(t) dt= \lim_{\e \to 0}\left(\int_{\dom\times\dom\setminus\Delta_\e} |x-y|^z dxdy+\sum_{j=0}^{k-1}\frac{\psi_\dom^{(j+1)}(0)}{j!}\frac{\e^{z+j+1}}{z+j+1}\right),
\end{equation}
where $k\in\mathbb N$ is such that $\R z>-k-1$, and ${\e^0}/{0}$ 
is to be replaced by $\log \e$ in case $z\in\mathbb Z,z <0$.
\end{definition}

As before, there is a meromorphic function $B_\dom(z)$ which coincides with $E_\dom(z)$ away from its poles, which are located at the negative integers $z=-k$ such that $\psi_\dom^{(k+1)}(0)\neq 0$. We call $B_\dom$ the {\em beta function} of $\dom$. As before 
\begin{equation}\label{pole_remove_domain}
  E_\dom(-k)=\lim_{z\to -k}\left(B_\dom(z)-\frac{1}{z+k}\Res(B_\dom,-k)\right)
\end{equation}
if $-k$ is a pole of $B_\dom$. Furthermore, the coefficients in \eqref{eq_def_E_dom} coincide with the residues of $B_\dom(z)$. Indeed, by \eqref{basic_residues},
\begin{equation}\label{basic_res_dom}
 \Res(B_\Omega,-k)=\frac{\psi_\Omega^{(k)}(0)}{(k-1)!}.
\end{equation}

In order to compute these residues, the following alternative approach, based on \eqref{eq_compact_bodies_z-energy_boundary_integral}, will be useful. Let $M=\partial\dom$, and $\rho\in C^\infty(M\times M)$  be given by $\rho(x,y)=\langle \uon{x},\uon{y}\rangle$. For $z\neq -n,-2$ 
\begin{equation}\label{second_approach}
 E_\dom(z)=\frac{-1}{(z+2)(z+n)}\mathrm{Pf.}\int_0^\infty t^{z+2} \psi'_{\rho}(t) dt.
\end{equation}
Indeed, for $\R z>-n$ this is \eqref{eq_compact_bodies_z-energy_boundary_integral}. For $z$ not a negative integer, the equality follows by analytic continuation. Finally, for $z\in \mathbb Z, {z<0}$, it follows from \eqref{pole_remove_domain}.

Note also that $B_\dom(z)=-P_\dom(z)$ for $-n-1<\R z<-n$, so the beta function is the analytic continuation of both the Riesz energy and minus the fractional perimeter.

Another consequence of \eqref{second_approach}, combined with Proposition \ref{even_odd} (ii), is the following:
\begin{proposition}\label{prop_residue_B_Omega}
The beta function $B_\dom(z)$ can have poles only at $z=-n$ and $z=-n-2j-1$ with $j\in\mathbb Z, j\geq 0$.
\end{proposition}

%!!!!!!!!!!!!!!!!!!!!!!!!!!!!!!!!!!!!!!!!!!!!!!!!!!!!!!!!!!!!!!!!
\subsection{Residues}
Next we compute the residues of the beta function, and we derive some explicit presentations of $E_\dom(z)$ in low dimensions.
\begin{lemma}
For $n>2$, the pole of $B_\dom(z)$ at $z=-n$ is simple and has residue
\begin{equation}
 \displaystyle \Res(B_\dom,-n)= \frac{1}{n-2}\int_{ M\times M}|x-y|^{2-n}\langle\uon{x},\uon{y}\rangle\,dxdy%={o_{n-1}}\text{\rm Vol}\,(\dom)  
, \label{formula_volume}
\end{equation}where  $o_k$ is the volume of the unit $k$-sphere in $\RR^{k+1}$.
For $n=2$ this pole is simple with residue
\begin{equation}
\Res(B_\dom,-2)=-\displaystyle \int_{ M\times M}\log|x-y|\,\langle\uon{x},\uon{y}\rangle\,dxdy%=2\pi A(\dom)
. \label{formula_area}
\end{equation}
\end{lemma}
%!!!!!!!!!!!!!!!!!!!!!!!!!!!!!!!!!!!!!!!!!!!!!!!!!!!!!!!!!!!!!!!!
%
\begin{proof} Equality \eqref{formula_volume} follows from Lemma \ref{lemma_Riesz_energy_compact_bodies_boundary_integral}. Let us prove  \eqref{formula_area}. 
Since $\int_{ M}\langle \uon{x},\uon{y}\rangle dy$ vanishes, we have
\[
\begin{array}{rcl}
\Res(B_\dom,-2)&=&\displaystyle \lim_{z\to-2}(z+2)B_\dom(z)\\[2mm]
&=&\displaystyle \lim_{z\to-2}\left(-\int_{ M\times M}\frac{|x-y|^{z+2}}{z+2}\langle\uon{x},\uon{y}\rangle dxdy \right)\\[4mm]
&=&\displaystyle -\lim_{z\to-2}\int_{ M\times M}\frac{|x-y|^{z+2}-1}{z+2}\langle\uon{x},\uon{y}\rangle dxdy \\[4mm]
&=&\displaystyle -\int_{ M\times M}\log|x-y|\langle\uon{x},\uon{y}\rangle dxdy,
\end{array}
\]
by dominated convergence (as $\log|y-x|$ is integrable on $M\times M$).
\end{proof}
\begin{corollary}
When $n>2$ the volume of a compact domain $\dom$ with boundary $M=\partial\dom$ is given by 
\begin{equation*}\label{formula_volume_boundary_integral}
\text{\rm Vol}\,(\dom)=\frac{1}{(n-2)o_{n-1}}\int_{ M\times M}|x-y|^{2-n}\langle\uon{x},\uon{y}\rangle\,dxdy.
\end{equation*}

When $n=2$ the area of a compact domain  $\Omega$ with boundary $M=\partial\dom$ is given by 
\begin{equation*}\label{formula_area_boundary_integral}
A(\dom)=-\frac1{2\pi} \int_{M\times M}\log|x-y|\,\langle\uon{x},\uon{y}\rangle\,dxdy.
\end{equation*}
\end{corollary}
\begin{proof}
The equations \eqref{basic_res_dom} and \eqref{expansion_psi} imply 
\begin{equation*}\label{residue_domain_volume}
\Res(B_\dom,-n)={o_{n-1}}\text{\rm Vol}\,(\dom)
\end{equation*}
for any $n\ge2$. 
The two formulae in the corollary follow from \eqref{formula_volume} and \eqref{formula_area}. 

{We remark that the two formulae can also be proved directly by application of the Stokes theorem and Hadamard-type regularization at $M$. }
\end{proof}

%
%
%!!!!!!!!!!!!!!!!!!!!!!!!!!!!!!!!!!!!!!!!!!!!!!!!!!!!!!!!!!!!!!!!
By \eqref{second_approach}, the other residues are given by 
\begin{equation*}\label{residues_domain}
 \Res(B_\dom,-n-2j-1)=\frac{-1}{(n+2j-1)!\,(2j+1)}\int_{ M}  \psi^{(n+2j-1)}_{\rho,x}(0)\, dx.
\end{equation*}
\begin{proposition}Let $\dom\subset \mathbb R^n$ be a compact domain bounded by a smooth hypersurface $ M$.
Given $x\in  M$, 
let $\rho(y)=\langle  \uon{x},\uon{y}\rangle$. Then
\begin{equation}\label{eq_psi_rho}
 \psi_{\rho,x}(t)=\frac{o_{n-2}t^{n-1}}{n-1}\left(1-\frac{t^2}{8(n+1)}\left(3(n-1)^2H^2-4K\right) +O(t^4)\right),
\end{equation}
where $H=\frac{1}{n-1}\sum_i k_i$ is the mean curvature, $K=\sum_{i<j}k_ik_j$ is the scalar curvature, and $k_1,\ldots,k_{n-1}$ are the principal curvatures of $ M$. Hence
\begin{align*} \label{residues_j01}
&\Res(B_\dom,-n)=o_{n-1}{\rm Vol}(\dom),\\
 &\Res(B_\dom,-n-1)=-\frac{o_{n-2}}{n-1}\textrm{\rm Vol}( M),
 \\  &\Res(B_\dom,-n-3)=\frac{o_{n-2}}{24(n^2-1)}\int_{ M} (3(n-1)^2H^2-4K) dx.
\end{align*}
\end{proposition}

\begin{proof}We can choose orthogonal coordinates $(v_1,\ldots,v_n)$ so that $x$ is the origin and $M=\partial\Omega$ coincides locally with the graph of a smooth function $g(v_1,\ldots,v_{n-1})$ and $\mathbf \uon_x=(0,\ldots,0,1)$. Using polar coordinates $(r,u)\in \mathbb R_{\geq 0}\times S^{n-2}$ in the domain of $g$, we parametrize the points $y\in M$ around $x$ by
\[y=h(r,u)=\left(r u_1, \ldots, r u_{n-1}, g(r\cdot u)\right)= \left(r u_1, \ldots, r u_{n-1}, -\frac{r^2}{2}k_n(u)+O(r^3)\right),
\]
where $k_n(u)=\sum_{i=1}^{n-1}k_iu_i^2$ is the normal curvature in the direction $u$. 
It is geometrically clear that
\[
h^\ast\left(\langle\uon{x},\uon{y}\rangle dy\right)=dv_1\cdots dv_{n-1}=r^{n-2}drdu.
\]
On the other hand, the distance between $x$ and $y$ is given by
\[
 t=t(r,u)=\sqrt{r^2+\frac{r^4}{4}k_n(u)^2+O(r^5)}=r\left(1+\frac12\frac{k_n(u)^2}{4}r^2+O(r^3)\right)
\]
Then, it follows that $r=r(t,u)$ can be expanded in a series of $t$ as 
\begin{equation*}\label{formula_r-t}
r=t\left(1-\frac18k_n(u)^2t^2+O(t^3)\right)
\end{equation*}
Now, using $(t,u)$ as coordinates instead of $(r,u)$, the area element of the plane $\{v_n=0\}$  can be expressed as 
\begin{align*}
r^{n-2}drdu&= t^{n-2}\left(1-\frac18k_n(u)^2t^2+O(t^3)\right)^{n-2}\left(1-\frac38k_n(u)^2t^2+O(t^3)\right)dtdu \\
&= t^{n-2}\left(1-\frac{n+1}{8}k_n^2(u)t^2+O(t^3)\right)dtdu.
\end{align*}
Therefore
\begin{align*}
\psi_{\rho,x}(\e)&=\int_{ M\cap B_\e(x)}\langle\uon{x},\uon{y}\rangle dy \\
&= \int_0^\e\int_{S^{n-2}} t^{n-2}\left(1-\frac{n+1}{8}\left(k_n(u)\right)^2t^2+O(t^3)\right)du dt\\
&= \frac{o_{n-2}\,\e^{n-1}}{n-1}-\frac{\e^{n+1}}{8}\int_{S^{n-2}} k_n(u)^2du +O(\e^{n+2}).
\end{align*}
By Proposition \ref{even_odd} we know that $\psi_{\rho,x}$ extends to an even (resp. odd) function when $n-1$ is even (resp. odd), so the latter $O(\e^{n+2})$ is in fact $O(\e^{n+3})$.

Finally, using e.g. \cite[Formula (A.5)]{gray} one gets
\[
 \int_{S^{n-2}}k_n(u)^2 du=\frac{o_{n-2}}{(n-1)(n+1)}\left(3\sum_{i=1}^{n-1} k_i^2+2\sum_{1\leq i<j\leq n-1}k_ik_j\right)
\]
where $k_1,\ldots, k_{n-1}$ are the principal curvatures of $M$ at $x$. Equation \eqref{eq_psi_rho} follows.
\end{proof}

\begin{theorem}Let $\dom\subset\mathbb R^n$ be a compact domain with smooth boundary $\partial \dom$. The first three poles (along the negative real axis) of $B_{\dom}(z)$  have the following residues
               \begin{enumerate}
\item  For $n=2$ 
\[
R_\dom(-2)=2\pi A(\dom),\>   
R_\dom(-3)=-2L( \partial \dom),\>   
R_\dom(-5)=\frac1{12}\int_{ \partial\dom}\kappa^2\,dx,
\]
where $L$ and $A$ denote length and area respectively, and
$\kappa$ denotes the curvature of $\partial\dom$.
\item For $n=3$
\begin{equation*}\label{residues_compact_body_dim3}
R_\dom(-3)=4\pi V(\dom), \>  
R_\dom(-4)=-\pi A(\partial \dom), \>  
R_\dom(-6)=\frac{\pi}{24}\int_{ M}(3H^2-K)dx,
\end{equation*}
where $V$ and $A$ denote volume and area respectivley, and $H,K$ are the mean and the Gauss curvatures of $ M$. 
\item For $n=4$
\begin{equation*}\label{residues_compact_body_dim4}
R_\dom(-4)={2}\pi^2 V_4(\dom), \>  
R_\dom(-5)=-\frac43\pi V_3(\partial\dom), \>  
R_\dom(-7)=\frac{\pi}{90}\int_{\partial\dom}(27H^2-4K)dx,
\end{equation*}
where $V_k$ denotes $k$-dimensional volume, and $H,K$ are the mean and scalar curvatures of $\partial \dom$.
               \end{enumerate}
\end{theorem}
%!!!!!!!!!!!!!!!!!!!!!!!!!!!!!!!!!!!!!!!!!!!!!!!!!!!!!!!!!!!!!!!!
%

The previous formulas allow to describe explicitly the $z$-energy for $\R z> -n-5$ in dimensions $n=2,3,4$ using \eqref{eq_def_E_dom} and \eqref{basic_res_dom}. Next we carry this out for $z=-2n$. 
\begin{corollary}Let $\dom\subset \RR^n$ be a compact domain with smooth boundary.
\begin{enumerate}
 \item For $n=2$, the regularized $(-4)$-energy is
\[
E_\dom(-4)=B_\dom(-4)
=\lim_{\e\rightarrow 0}\left(\int_{\dom\times\dom\setminus\Delta_\e}\frac{dxdy}{|x-y|^4}-\frac{\pi}{\e^2}A(\dom)+\frac{2}{\e}L(\partial\dom)\right).
\]
\item For $n=3$ the regularized $(-6)$-energy is 
\[
\begin{array}{rcl}
\displaystyle 
\E{\dom}{-6}
&=& \displaystyle 
\lim_{z\to-6}\left(B_\dom(z)-\frac\pi{24(z+6)}\int_{ \partial\dom}(3H^2-K)dx\right) \\[4mm]
&=&\displaystyle \lim_{\e\rightarrow 0}\left(\int_{\dom\times\dom\setminus\Delta_\e}\frac{dxdy}{|x-y|^6}
-\frac{4\pi }{3\e^3}\textrm{\rm Vol}(\dom)+\frac{\pi }{2\e^2}A( \partial\dom)+\frac{\pi\log\e}{24}\int_{ \partial\dom}(3H^2-K)dx
\right).
\end{array}
\]
\item For $n=4$, the regularized $(-8)$-energy is 
\[\begin{array}{l}
 E_\dom(-8)=B_\dom(-8) \\
\displaystyle =\lim_{\e\rightarrow 0}\left(\int_{\dom\times\dom\setminus\Delta_\e}\frac{dxdy}{|x-y|^8}-\frac{\pi^2}{{2}\e^4} {V_4(\dom)} +\frac{4\pi}{9\e^3} {V_3( \partial\dom)} -\frac{\pi}{90\e}\int_{ \partial\dom}(27H^2-4K)dx\right).
\end{array}
\]
\end{enumerate}
\end{corollary} 
In  \cite{OS}, we introduced an energy $E(\dom)$ for planar compact domains $\dom\subset\mathbb R^2$. This energy is related to $E_\dom(-4)$ by $E(\dom)=E_\dom(-4)+\frac{\pi^2}{4}\chi(\dom)$. Indeed by \cite[Definition 3.11 and Proposition 3.13]{OS}, one has 
\begin{equation*}\label{E_OS}
{{E}}(\dom)=\lim_{\e\rightarrow 0}\left(\int_{\dom\times\dom\setminus\Delta_\e}\frac{dxdy}{|x-y|^4}-\frac{\pi}{\e^2}A(\dom)+\frac{2}{\e}L(\partial\dom)\right)+\frac{\pi^2}{4}\chi(\dom).
\end{equation*}
It was shown in \cite{OS}, that this energy is M\"obius invariant. In the next section we prove the analogous result for any even dimension.

\subsection{M\"obius invariance}
\begin{proposition}\label{last_proposition}
Under a homothety $x\mapsto cx$ $(c>0)$, the residues of the Riesz energy behave as follows. 
\[
\begin{array}{rcl}
R_{c\Omega}(-k)&=&\displaystyle c^{2n-k} R_\Omega(-k) \hspace{0.7cm}(k\ge n). \\[2mm]
\E{c\Omega}{z}&=&\displaystyle c^{2n+z}\left(E_\Omega(z)+(\log c) R_\Omega(z)\right). 
\end{array}
\]
Hence the regularized $z$-energy is not scale invariant if $z\ne -2n$. 
The regularized $(-2n)$-energy is scale invariant if and only if $R_\Omega(-2n)$ vanishes for any $\Omega$. 
\end{proposition}
\proof The arguments in Lemma \ref{lemma_residue_homothety} and Proposition \ref{proposition_energy_homothety} go parallel here. \endproof
\begin{example} \rm Let $\Omega=B^n$ be the $n$-dimensional unit ball. Using Lemma \ref{lemma_Riesz_energy_compact_bodies_boundary_integral} one easily gets the following expression (which appears also in \cite{Mi,HR})
%%%%%%%%%%%%%%%%%%%%%%%%%%%%%%%%%%%%%%%%%%%%%%%%%%%%%%%%%%%%%%%%%%%%%%%%%%%%%%%%
\begin{eqnarray}
B_{\dom}(z)&=&\frac{2^{z+n}o_{n-1}o_{n-2}}{(n-1)(z+n)}\,B\!\left(\frac{z+n+1}2,\frac{n+1}2\right)   
\nonumber \\
&=&\displaystyle \left\{
\begin{array}{ll}
\displaystyle \frac{2^{z+n+1}\,\pi^{n-\frac12}\,\Gamma\left(\frac{z}2+\frac{n+1}2\right)}{(z+n)\left(\frac{n}2-1\right)!\,\Gamma\left(\frac{z}2+{n+1}\right)} & \hspace{0.5cm}(\mbox{ if $n$ is even}) \\[6mm]
\displaystyle \frac{2^{z+2n+1}\,\pi^{n-1}}{(z+n)\,(n-2)!!\,(z+n+1)(z+n+3)\cdots(z+2n)} & \hspace{0.5cm}(\mbox{ if $n$ is odd}),  
\end{array}
\right.\nonumber \label{unit_ball_even_odd}
\end{eqnarray}
where $(n-2)!!=(n-2)\cdot(n-4)\cdots3\cdot1$. Hence, the beta function of a ball has infinitely many poles at $z=-n, -n-1, -n-3, -n-5, \dots$ when $n$ is even, and exactly $(n+3)/2$ poles at $z=-n, -n-1, -n-3, \dots, -2n$ when $n$ is odd. 
\end{example}
\begin{theorem}\label{thm4.11}
The regularized $z$-energy $E_\dom(z)$ is a M\"obius invariant if and only if $n=\dim\dom$ is even and $z=-2n$.
\end{theorem}
\begin{proof}
The regularized $z$-energy is scale invariant only if $z=-2n$ by Proposition \ref{last_proposition}. 
The example above shows that the regularized $(-2n)$-energy is not scale-invariant if $n$ is odd. 
Propositions \ref{prop_residue_B_Omega} and \ref{last_proposition} show that $E_\dom(-2n)$ is scale invariant if $n$ is even. 
Therefore, we have only to show that $E_\dom(-2n)=E_{I(\dom)}(-2n)$ if $n$ is even, $I$ is an inversion with respect to the unit sphere, and $\dom$ is a compact domain in $\RR^n$ with smooth boundary that does not contain the origin. 

For $\mathfrak{Re}(z)>-n$, and denoting $\widetilde\Omega=I(\Omega)$  we have
\[
E_{\widetilde\Omega}(z)-E_\Omega(z)=\int_{\Omega\times \Omega} |x-y|^z\rho_z(x,y) dxdy
\]
where $\rho_z(x,y)=|x|^{-z-2n}|y|^{-z-2n} -1$. By Proposition \ref{uniform}, we have $\Psi_{\rho_z}\in C^\infty([0,d\,])$ and all derivatives $\Psi^{(k)}_{\rho_z}$ converge uniformly to $0$ as $z\to -2n$. Arguing as in Proposition 3.3, 
\begin{align*}
E_{\widetilde\Omega}(z)-E_\Omega(z)=\int_0^d t^z\Psi'_{\rho_z}(t)dt+\int_{(\Omega\times \Omega)\setminus \Delta_d} |x-y|^z\rho_z(x,y)dxdy
\end{align*}
for $\mathfrak{Re}(z)>-n$. For $\mathfrak{Re}z>-2n-1$ we have
\begin{equation}\label{F_threeterms}
\begin{array}{rl}
E_{\widetilde\dom}(z)-E_\dom(z)=&\displaystyle \int_0^d t^z\left[\Psi_{\rho_z}'(t)-\sum_{j=0}^{2n-1}\frac{\Psi_{\rho_z}^{(j+1)}(t)}{j!}\,t^j
\right]\,dxdt  +\sum_{k=1}^{2n}\frac{\Psi_{\rho_z}^{(k)}(0)\, d^{z+k}}{(k-1)!\,(z+k)}\\[6mm]
&\displaystyle +\int_{(\dom\times \dom)\setminus \Delta_d} |x-y|^z \rho_{z}(x,y)dxdy. 
\end{array}
\end{equation}

The third term of the right hand side of \eqref{F_threeterms} goes to $0$ as $z$ tends to $-2n$ since $\rho_z(x,y)$ goes to $0$ uniformly. 

The {modulus} of the first term goes to $0$ as $z$ tends to $-2n$ since it is bounded above by 
\[
\frac1{(2n)!}\sup_{0\le t\le d}\left|\Psi_{\rho_z}^{(2n+1)}(t)\right|\left|\int_0^d t^{z+2n}\,dt\right|
=\frac{ |d^{z+2n+1} |}{(z+2n+1)(2n)!}\sup_{0\le t\le d}\left|\Psi_{\rho_z}^{(2n+1)}(t)\right|,
\]
which tends to $0$ as $z$ goes to $-2n$ by Proposition \ref{uniform}. 

By Proposition 4.4, the function $E_{\widetilde\dom}(z)-E_{\dom}(z)$ has possible poles at $z=-n$ and $z=n-(2j+1)$ with  $j\in\mathbb Z, j\geq 0$. 
Since $n$ is even,  it does not have a pole at $z=-2n$. 
Hence, the term $k=2n$ in \eqref{F_threeterms} must vanish identically. It follows by Proposition \ref{uniform} that the sum over $k$ in \eqref{F_threeterms}, and therefore $E_{\widetilde\dom}(z)-E_\dom(z)$, tends to $0$ as $z$ approaches $-2n$.

This completes the proof of the M\"obius invariance of $E_\dom(-2n)$. 
\end{proof}

%%%%%%%%%%%%%%%%%%%%%%%%%%%%%%%%%%%%%%%%%%%%%%%%%%%%%%%%%%%%%%%%%%%%%%%%%%%%%%%%%%%%%%

\end{document}